\documentclass{amsart}

\usepackage{amsfonts, amssymb, amsmath, eucal, verbatim, amsthm, amscd, enumerate}
\usepackage{graphicx}

\newtheorem{theorem}{Theorem}
\newtheorem{lemma}[theorem]{Lemma}
\newtheorem{corollary}[theorem]{Corollary}

\newtheorem{remark}[theorem]{Remark}

\theoremstyle{definition}

\theoremstyle{remark}

\numberwithin{equation}{section}

\def\[{{\langle}}
\def\]{{\rangle}}
\def\R{{\mathbb  R}}

\def\N{{\mathbb N}}

\def\d{{\mbox{div}}}

\def\A{{\alpha}}

\begin{document}

\title{Closure properties of solutions to heat inequalities}

\author{Jonathan Bennett}
\author{Neal Bez}
\thanks{Both authors were supported by EPSRC
grant EP/E022340/1. \\ \emph{Acknowledgement}. The first author
would like to thank Tony Carbery and Terence Tao for many
interesting and useful conversations on the context
and perspective of this work. We would also like to thank Dirk Hundertmark for his
contribution to the early stages of this work.}

\address{Jonathan Bennett and Neal Bez\\ School of Mathematics \\
The University of Birmingham \\
The Watson Building \\
Edgbaston \\
Birmingham \\
B15 2TT \\
United Kingdom}\email{J.Bennett@bham.ac.uk\\ N.Bez@bham.ac.uk}

\subjclass[2000]{44A35, 35K99, 52A40}

\date{28th of May 2008}



\begin{abstract} We prove that if $u_1,u_2 : (0,\infty) \times \R^d
\rightarrow (0,\infty)$ are sufficiently well-behaved solutions to
certain heat inequalities on $\R^d$ then the function $u: (0,\infty) \times \R^d
\rightarrow (0,\infty)$ given by $u^{1/p}=u_1^{1/p_1}
* u_2^{1/p_2}$ also satisfies a heat inequality of a similar type
provided $\tfrac{1}{p_1} + \tfrac{1}{p_2} = 1 + \tfrac{1}{p}$. On
iterating, this result leads to an analogous statement concerning
$n$-fold convolutions. As a corollary, we give a direct heat-flow
proof of the sharp $n$-fold Young convolution inequality and its
reverse form.
\end{abstract}

\maketitle

\section{Introduction}
It is known that if $d\in\mathbb{N}$ and
$u_1,u_2:(0,\infty)\times\mathbb{R}^d\rightarrow (0,\infty)$ satisfy
the heat inequality
\begin{equation}\label{first}
\partial_t u\geq\frac{1}{4\pi}\Delta
u,\end{equation} then any geometric mean of $u_1$ and $u_2$
satisfies the same heat inequality; i.e. for $1\leq
p_1,p_2\leq\infty$ satisfying $\tfrac{1}{p_1}+\tfrac{1}{p_2}=1$, the
function $$u:=u_1^{1/p_1}u_2^{1/p_2}$$ also satisfies the heat inequality
\eqref{first}. As a corollary to this, provided that $u_1$ and $u_2$
are sufficiently well-behaved, by the divergence theorem it follows
that the quantity
$$ Q(t):=\int_{\mathbb{R}^d}u_1(t,x)^{1/p_1}
u_2(t,x)^{1/p_2}dx$$ is nondecreasing for all $t>0$. Furthermore, on
insisting that, for $j=1,2$, $u_j$ satisfies \eqref{first} with
equality and sufficiently well-behaved initial data $f_j^{p_j}$, it
follows from this monotonicity that
$$
\int_{\mathbb{R}^d}f_1(x)f_2(x)dx=\lim_{t\rightarrow 0}Q(t)\leq
\lim_{t\rightarrow\infty}Q(t)=\|f_1\|_{p_1}\|f_2\|_{p_2};
$$
that is, we recover the classical H\"older inequality.

This closure property of solutions to heat inequalities may be
generalised considerably. Let $m,d\in\mathbb{N}$,
$d_1,\hdots,d_m\in\mathbb{N}$, $1\leq p_1,\hdots,p_m\leq\infty$ and
for each $1\leq j\leq m$ let
$B_j:\mathbb{R}^d\rightarrow\mathbb{R}^{d_j}$ be such that
$B_j^*B_j$ is a projection and
$$
\sum_{j=1}^m\tfrac{1}{p_j}B_j^*B_j=I_d,$$ where $I_d$ denotes the
identity on $\mathbb{R}^d$. If $u_j:(0,\infty)\times
\mathbb{R}^{d_j}\rightarrow (0,\infty)$ satisfies the heat
inequality \eqref{first} for each $j$, then one may show that the
same is true of the ``geometric mean"
$$u(t,x):=\prod_{j=1}^mu_j(t,B_jx)^{1/p_j}.$$ (Here $\Delta$
acts in the number of variables dictated by context.) Very much as
before, an immediate consequence of this is that $\int u(t,\cdot)$
is nondecreasing for all $t>0$, and that from this monotonicity one
may deduce the inequality
$$
\int_{\mathbb{R}^d}\prod_{j=1}^mf_j(B_jx)dx\leq
\prod_{j=1}^m\|f_j\|_{L^{p_j}(\mathbb{R}^{d_j})}.$$ This is the
celebrated geometric Brascamp--Lieb inequality due to Ball
\cite{Ball} for rank-one projections and Barthe \cite{Barthetrans2}
in the general rank case. Such a heat-flow approach to proving
inequalities, by its nature, generates sharp constants and
guarantees the existence of centred gaussian extremisers. All of
these observations were first made by Carlen, Lieb and Loss
\cite{CLL} for rank-one projections and by Bennett, Carbery, Christ
and Tao \cite{BCCT} in the general rank case. Recently, Barthe and
Huet \cite{BartheHuet} have given a different heat-flow proof of the
geometric Brascamp--Lieb inequality and, moreover, the same line of
argument also led them to a heat-flow proof of Barthe's reverse form
of the geometric Brascamp--Lieb inequality. See \cite{Barthetrans1},
\cite{Barthetrans2} for a statement of this reverse inequality. The
reader is referred to \cite{Barthe2} and the references therein for
further discussion of heat-flow methods in the context of such geometric
inequalities.

Aside from the geometric means above, and the trivial operation of ordinary addition, it is not difficult to verify that harmonic addition also preserves the set of solutions of \eqref{first}; i.e. if $u_1,u_2:(0,\infty)\times\mathbb{R}^d\rightarrow (0,\infty)$ satisfy \eqref{first}, then the function
$u$ given by
$$
\frac{1}{u}=\frac{1}{u_1}+\frac{1}{u_2}$$
also satisfies \eqref{first}. By the divergence theorem, this closure property is easily seen to imply a variant of the triangle inequality for harmonic addition.

All of the above closure properties involve pointwise operations.
The main purpose of this article is to establish closure properties
under rather different operations involving convolution. As a
consequence we provide heat-flow proofs of sharp convolution
inequalities that do not proceed via duality.

\subsection{Main Results}

Let $d \in \N$ and suppose $0<p_1,p_2,p<\infty$ satisfy
\begin{equation} \label{e:scaling}
  \frac{1}{p_1} + \frac{1}{p_2} = 1 + \frac{1}{p}.
\end{equation}
Let $0\leq \sigma_1,\sigma_2<\infty$ satisfy
\begin{equation} \label{e:BL}
  \frac{1}{p_1}\left(1-\frac{1}{p_1}\right)\sigma_2 =
  \frac{1}{p_2}\left(1-\frac{1}{p_2}\right)\sigma_1.
\end{equation}

The main contribution in this paper is captured by the following. We
clarify that the operation $*$ will denote spatial convolution.

\begin{theorem} \label{t:closure} For $j=1,2$ suppose that $u_j:(0,\infty) \times
\mathbb{R}^d \rightarrow (0,\infty)$ is such that
$u_j(t,\cdot)^{1/p_j}$, $\partial_t(u_j(t,\cdot)^{1/p_j})$,
$\nabla(u_j(t,\cdot)^{1/p_j})$, $u_j(t,\cdot)^{1/p_j}|\nabla \log
u_j(t,\cdot)|^2$ and $\Delta(u_j(t,\cdot)^{1/p_j})$ are rapidly
decreasing in space locally uniformly in time $t>0$. Let $u :
(0,\infty) \times \R^d \rightarrow (0,\infty)$ be given by
$$ u^{1/p} := u_1^{1/p_1} * u_2^{1/p_2} $$
and let $\sigma := (\sigma_1p_1 + \sigma_2p_2)/r$. Then
$u(t,\cdot)^{1/p}$, $\partial_t(u(t,\cdot)^{1/p})$,
$\nabla(u(t,\cdot)^{1/p})$, $u(t,\cdot)^{1/p}|\nabla \log u(t,\cdot)
|^2$ and $\Delta(u(t,\cdot)^{1/p})$ are also rapidly decreasing in
space locally uniformly in time $t>0$. Furthermore,
\begin{enumerate}
  \item if $p_j \geq 1$ and
  \begin{equation*}
    \partial_t u_j \geq \frac{\sigma_j}{4\pi}
    \Delta u_j,
  \end{equation*}
  for $j=1,2$, then
  \begin{equation} \label{e:heatinuj}
    \partial_t u \geq
    \frac{\sigma}{4 \pi}\Delta u;
  \end{equation}
\item if $p_j\leq 1$ and
  \begin{equation*}
    \partial_t u_j \leq \frac{\sigma_j}{4\pi} \Delta
    u_j,
  \end{equation*}
  for $j=1,2$, then
\begin{equation} \label{e:heatinu}
    \partial_t u \leq
    \frac{\sigma}{4 \pi}\Delta u.
  \end{equation}
\end{enumerate}
\end{theorem}
An important feature of this closure property is that the
(technical) regularity ingredients are all satisfied when $u_1$ and
$u_2$ are solutions to heat \emph{equations} with sufficiently
well-behaved initial data. Indeed, we shall see that Theorem
\ref{t:closure} implies the following.
\begin{corollary} \label{c:Ymon}
For $j=1,2$ let $u_j$ satisfy the heat equation
\begin{equation*}
    \partial_t u_j = \frac{\sigma_j}{4\pi} \Delta u_j
\end{equation*}
with initial data a compactly supported positive finite Borel
measure. Let $Q : (0,\infty) \rightarrow (0,\infty)$ be given by
\begin{equation*}
  Q(t) := \|u_1(t,\cdot)^{1/p_1} * u_2(t,\cdot)^{1/p_2}\|_{L^p(\R^d)}.
\end{equation*}
If $p_1,p_2 \geq 1$ then $Q(t)$ is nondecreasing for each $t > 0$
and if $p_1,p_2 \leq 1$ then $Q(t)$ is nonincreasing for each $t >
0$.
\end{corollary}
The proofs of Theorem \ref{t:closure} and Corollary \ref{c:Ymon}
appear in Section \ref{section:closure1}.

\begin{remark} \label{remarks:intro} \emph{Under the hypotheses of Corollary
\ref{c:Ymon}, it follows from our proof in Section
\ref{section:closure1} that
\begin{align*}
  Q'(t) = &\frac{\varepsilon}{8\pi Q(t)^{p-1}} \int_{\R^d} \int_{\R^d}
  \int_{\R^d}
    (u_1(t,\cdot)^{1/p_1} * u_2(t,\cdot)^{1/p_2})(x)^{p-2}
  u_1(t,x-y)^{1/p_1}u_2(t,y)^{1/p_2} \times \\ &
  u_1(t,x-z)^{1/p_1}u_2(t,z)^{1/p_2}
  \left|\left(\tfrac{\sigma_1}{p_1}|\tfrac{1}{p_1}-1|\right)^{1/2}\tfrac{\nabla u_1}{u_1}(t,x-y)
+
  \left(\tfrac{\sigma_2}{p_2}|\tfrac{1}{p_2}-1|\right)^{1/2}\tfrac{\nabla u_2}{u_2}(t,y) \right.\\ & \left. \quad   -
  \left(\tfrac{\sigma_1}{p_1}|\tfrac{1}{p_1}-1|\right)^{1/2}\tfrac{\nabla u_1}{u_1}(t,x-z) -
  \left(\tfrac{\sigma_2}{p_2}|\tfrac{1}{p_2}-1|\right)^{1/2}\tfrac{\nabla u_2}{u_2}(t,z)
  \right|^2 \,dxdydz
\end{align*}
for each $t > 0$. Here $\varepsilon$ is defined to be 1 if $p_1,p_2
\geq 1$ and $-1$ if $p_1,p_2 \leq 1$. Consequently, if exactly one
of $p_1$ and $p_2$ is equal to 1 then the corresponding monotonicity
in Corollary \ref{c:Ymon} is strict. In this case and if $p_j$ is
equal to 1, then it is amusing to note that $\sigma_j$ is zero by
\eqref{e:BL}; that is, the heat-flow $u_j$ is constant in time.}
\end{remark}

We now describe the sharp convolution inequalities that follow from
these results. Recall that the sharp form of the Young convolution
inequality on $\R^d$ states that if $p_1,p_2\geq 1$ and
$\tfrac{1}{p_1} + \tfrac{1}{p_2} = 1 + \tfrac{1}{p}$ then
\begin{align} \label{e:sharpYoung}
    \|f_1 * f_2\|_{L^p(\R^d)} \leq \left(\frac{C_{p_1}C_{p_2}}{C_{p}}\right)^d
    \|f_1\|_{L^{p_1}(\R^d)}\|f_2\|_{L^{p_2}(\R^d)}
\end{align}
for any nonnegative functions $f_j$ in $L^{p_j}(\R^d)$, where $C_r
:= \big(r^{1/r}/r'^{1/r'}\big)^{1/2}$. The sharp constant in
\eqref{e:sharpYoung} is due to Beckner \cite{Beckner2},
\cite{Beckner} and Brascamp and Lieb \cite{BL}. The sharp reverse
form of \eqref{e:sharpYoung} states that if $p_1,p_2 \leq 1$ and
$\tfrac{1}{p_1} + \tfrac{1}{p_2} = 1 + \tfrac{1}{p}$ then
 \begin{align} \label{e:reverseYoung}
    \|f_1 * f_2\|_{L^p(\R^d)} \geq \left(\frac{C_{p_1}C_{p_2}}{C_{p}}\right)^d
    \|f_1\|_{L^{p_1}(\R^d)}\|f_2\|_{L^{p_2}(\R^d)}
  \end{align}
for any nonnegative functions $f_j$ in $L^{p_j}(\R^d)$. Leindler
\cite{Leindler} proved \eqref{e:reverseYoung} with a nonsharp
constant and Brascamp and Lieb found the sharp constant in \cite{BL}
(see also Barthe's simpler argument in \cite{Barthe} which proves
both forms with sharp constants). It is easy to see that from
Corollary \ref{c:Ymon} one may recover both \eqref{e:sharpYoung} and
\eqref{e:reverseYoung}. To see this, let $0<p_1,p_2,p<\infty$
satisfy \eqref{e:scaling} and note that it suffices to verify both
inequalities when the $f_j^{p_j}$ are bounded, integrable and
compactly supported functions. For $j=1,2$ let $u_j$ satisfy the
heat equation
\begin{equation} \label{e:heateqn}
    \partial_t u_j = \frac{\sigma_j}{4\pi}\Delta u_j
\end{equation}
with initial data $u_j(0,x) := f_j(x)^{p_j}$. By the dominated
convergence theorem, one can show that,
  \begin{equation*}
    \lim_{t \rightarrow 0} Q(t) = \|f_1 \ast f_2 \|_{L^p(\R^d)}
  \end{equation*}
and, combined with a simple change of variables,
\begin{equation*}
    \lim_{t \rightarrow \infty} Q(t) = \big\|H_{\sigma_1}^{1/p_1} * H_{\sigma_2}^{1/p_2} \big\|_{L^p(\R^d)}
    \|f_1\|_{L^{p_1}(\R^d)}\|f_2\|_{L^{p_2}(\R^d)}.
\end{equation*}
  Here,
  \begin{equation*}
  H_t(x) := (1/t)^{d/2}e^{-\pi|x|^2/t}
  \end{equation*}
  is the appropriate heat kernel at time $t$ and $\sigma_1, \sigma_2$ satisfy \eqref{e:BL}. A direct computation shows that
  \begin{equation*}
   \big\|H_{\sigma_1}^{1/p_1} * H_{\sigma_2}^{1/p_2} \big\|_{L^p(\R^d)} =
   \left(\frac{C_{p_1}C_{p_2}}{C_{p}}\right)^d
  \end{equation*}
  and hence, Corollary \ref{c:Ymon} immediately implies both
  \eqref{e:sharpYoung} and \eqref{e:reverseYoung}. We also remark that if the initial data for $u_1$ and $u_2$
  are extremal then $Q(t)$ is constant in time. It is possible to
  recover the complete characterisation of the extremals
in the Young convolution inequality and its reverse form from the
expression for $Q'(t)$ in Remark \ref{remarks:intro}. We omit the
details of this.

  An alternative perspective on the sharp Young convolution inequality on $\R^d$ is
  to consider the following dual formulation. Suppose $1\leq p_1,p_2,p_3 <\infty$
  satisfy $\tfrac{1}{p_1} + \tfrac{1}{p_2} + \tfrac{1}{p_3} = 2$. The inequality
  \begin{equation} \label{e:dual}
    \int_{\R^d}\int_{\R^d}
    f_1(x)^{1/p_1}f_2(y)^{1/p_2}f_3(x-y)^{1/p_3}dxdy \leq
    \prod_{j'=1}^3 C_{p_{j'}}^d \prod_{j=1}^3
    \|f_{j}\|_{L^1(\R^d)}^{1/p_j}
  \end{equation}
  for all nonnegative integrable functions $f_j$ is equivalent to the Young convolution
  inequality in \eqref{e:sharpYoung}. It is known that if each $f_j$ evolves under an appropriate
  heat-flow $u_j$ then the quantity
  \begin{equation*}
    \int_{\R^d} \int_{\R^d} u_1(t,x)^{1/p_1}u_2(t,y)^{1/p_2}u_3(t,x-y)^{1/p_3} \,dxdy
  \end{equation*}
  is nondecreasing for each $t > 0$ from which the inequality in
  \eqref{e:dual} follows. This type of dualised heat-flow approach to the Young convolution
inequality on $\R^d$ can be found in \cite{CLL} and \cite{BCCT}.
Carlen, Lieb and Loss have also shown that this type of heat-flow
approach can be used to prove certain analogues of the Young
convolution inequality in other settings. See \cite{CLL} and
\cite{CLL2} for analogues on the euclidean sphere and the
permutation group, respectively (see also \cite{BCM}).

It is also worth noting that in our undualised setup when the
exponent $p$ is a natural number and $1\leq p_1, p_2 <\infty$, by
multiplying out the $p$th power of the integral one may deduce the
monotonicity of $Q$ directly from \cite{BCCT} (see also \cite{CLL}).

\subsection{Iterated convolutions}

Naturally Theorem \ref{t:closure} self-improves to a
result involving iterated convolutions, which we now state. Let
$0<p_1,\ldots,p_n,p<\infty$ satisfy
  \begin{equation} \label{e:scalinggen}
    \sum_{j=1}^n \frac{1}{p_j} = n-1+\frac{1}{p}
  \end{equation}
and let $0 \leq \sigma_1,\ldots,\sigma_n < \infty$ satisfy
\begin{equation*}
  \frac{1}{p_j}\left(1-\frac{1}{p_j}\right)\sigma_k =
  \frac{1}{p_k}\left(1-\frac{1}{p_k}\right)\sigma_j
\end{equation*} for each $j,k=1,\ldots,n$. As before these relations
uniquely define $\sigma_1,\hdots,\sigma_n$ up to a common scale factor.

\begin{corollary} \label{c:closure} For $j=1, \ldots, n$ suppose that $u_j:(0,\infty) \times
\mathbb{R}^d \rightarrow (0,\infty)$ is such that
$u_j(t,\cdot)^{1/p_j}$, $\partial_t(u_j(t,\cdot)^{1/p_j})$,
$\nabla(u_j(t,\cdot)^{1/p_j})$, $u_j(t,\cdot)^{1/p_j}|\nabla \log
u_j(t,\cdot)|^2$ and $\Delta(u_j(t,\cdot)^{1/p_j})$ are rapidly
decreasing in space locally uniformly in time for $t>0$. Let $u :
(0,\infty) \times \R^d \rightarrow (0,\infty)$ be given by
$$ u^{1/p} := u_1^{1/p_1} * \cdots * u_n^{1/p_n} $$
and let $\sigma := \tfrac{1}{r}\sum_{j=1}^n \sigma_jp_j$. Then
$u(t,\cdot)^{1/p}$, $\partial_t(u(t,\cdot)^{1/p})$,
$\nabla(u(t,\cdot)^{1/p})$, $u(t,\cdot)^{1/p}$ $|\nabla \log
u(t,\cdot)|^2$ and $\Delta(u(t,\cdot)^{1/p})$ are also rapidly
decreasing in space locally uniformly in time for $t>0$.
Furthermore,
\begin{enumerate}
  \item if $p_j \geq 1$ and
  \begin{equation*}
    \partial_t u_j \geq \frac{\sigma_j}{4\pi}
    \Delta u_j,
  \end{equation*}
  for $j=1,\ldots,n$, then
  \begin{equation*} \label{e:teapot}
    \partial_t u \geq
    \frac{\sigma}{4 \pi}\Delta u;
  \end{equation*}
\item if $p_j\leq 1$ and
  \begin{equation*}
    \partial_t u_j \leq \frac{\sigma_j}{4\pi} \Delta
    u_j,
  \end{equation*}
  for $j=1,\ldots,n$, then
\begin{equation*} \label{e:heatinequality2}
    \partial_t u \leq
    \frac{\sigma}{4 \pi}\Delta u.
  \end{equation*}
\end{enumerate}
\end{corollary}

It is a simple exercise to verify that Corollary \ref{c:closure}
follows from Theorem \ref{t:closure}.

\begin{corollary} \label{cor:iterated}
For $j=1,\ldots,n$ let $u_j$ satisfy the heat equation
\begin{equation*}
    \partial_t u_j = \frac{\sigma_j}{4\pi} \Delta u_j
\end{equation*}
with initial data a compactly supported positive finite Borel
measure. Let $Q : (0,\infty) \rightarrow (0,\infty)$ be given by
\begin{equation*}
  Q(t) := \|u_1(t,\cdot)^{1/p_1} * \cdots * u_n(t,\cdot)^{1/p_n}\|_{L^p(\R^d)}.
\end{equation*}
If $p_1,\ldots,p_n \geq 1$ then $Q(t)$ is nondecreasing for each $t
> 0$ and if $p_1,\ldots,p_n \leq 1$ then $Q(t)$ is nonincreasing for each $t >
0$.
\end{corollary}
We remark that Corollary \ref{cor:iterated} follows from Corollary
\ref{c:closure} in the same way that Corollary \ref{c:Ymon} follows
from Theorem \ref{t:closure}. As one may expect, from Corollary
\ref{cor:iterated} (and its proof) we recover the sharp $n$-fold
Young convolution inequality, its reverse form and a complete
characterisation of extremals. We omit the details of this.

When $p$ is an even integer, this $n$-fold Young convolution
inequality is of course related to the Hausdorff--Young inequality
via Plancherel's theorem. In particular,
\begin{equation} \label{e:FT}
  \big\| \widehat{u(t,\cdot)^{1/p'}}
    \big\|_{L^{p}(\R^d)} = \|u(t,\cdot)^{1/p'} * \cdots *
    u(t,\cdot)^{1/p'}\|_{L^2(\R^d)}^{2/p}
\end{equation}
where \hspace{0.01cm} $\widehat{\;}$ \hspace{0.01cm} denotes the
Fourier transform and the iterated convolution is $p/2$-fold. By
Corollary \ref{cor:iterated} it follows that the above quantity is
nondecreasing for $t > 0$ if $u$ satisfies the heat equation
$\partial_t u = \tfrac{1}{4\pi}\Delta u$ with initial data a
compactly supported finite positive Borel measure. We remark that
the nondecreasingness of the quantity in \eqref{e:FT} also follows
from \cite{BCCT}. We refer the interested reader to \cite{BBC} for
an explicit verification of how this fact follows from \cite{BCCT}
and for counterexamples to the monotonicity of $$\big\|
\widehat{u(t,\cdot)^{1/p'}} \big\|_{L^{p}(\R^d)}$$ whenever $p$ is
not an even integer.

\subsection{Further results}

We now describe some extensions of our results when the scaling
condition \eqref{e:scaling} (or more generally \eqref{e:scalinggen})
is relaxed. Let $0<\alpha_1, \alpha_2\leq 1$ and $1\leq p <\infty$
be such that
\begin{equation*}
  \alpha_1 + \alpha_2 \geq 1 + \frac{1}{p}.
\end{equation*}
Suppose $0\leq \rho_1, \rho_2 \leq 1$ satisfy
\begin{equation} \label{e:weight}
  \rho_1\alpha_1 + \rho_2\alpha_2 = 1+\frac{1}{p}
\end{equation}
and let $0\leq \sigma_1,\sigma_2<\infty$ satisfy
\begin{equation} \label{e:BL2}
  \alpha_1(1-\rho_1\alpha_1)\sigma_2 =
  \alpha_2(1-\rho_2\alpha_2)\sigma_1.
\end{equation}

\begin{theorem} \label{t:closure2}
  For $j=1,2$ suppose that $u_j:(0,\infty) \times
\mathbb{R}^d \rightarrow (0,\infty)$ is such that
$u_j(t,\cdot)^{\A_j}$, $\partial_t(u_j(t,\cdot)^{\A_j})$,
$\nabla(u_j(t,\cdot)^{\A_j})$, $u_j(t,\cdot)^{\A_j}|\nabla \log
u_j(t,\cdot)|^2$ and $\Delta(u_j(t,\cdot)^{\A_j})$ are rapidly
decreasing in space locally uniformly in time for $t>0$. Let $u :
(0,\infty) \times \R^d \rightarrow (0,\infty)$ be given by
$$ u(t,x)^{1/p} := t^{d(\alpha_1+\alpha_2-1-1/p)/2}(u_1(t,\cdot)^{\A_1} *
    u_2(t,\cdot)^{\A_2})(x) $$
and let $\sigma := (\tfrac{\sigma_1}{\A_1} +
  \tfrac{\sigma_2}{\A_2})/p$. Then
$u(t,\cdot)^{1/p}$, $\partial_t(u(t,\cdot)^{1/p})$,
$\nabla(u(t,\cdot)^{1/p})$, $u(t,\cdot)^{1/p}|\nabla \log
u(t,\cdot)|^2$ and $\Delta(u(t,\cdot)^{1/p})$ are also rapidly
decreasing in space locally uniformly in time for $t>0$.

Furthermore, if
\begin{equation} \label{e:he}
  \partial_t u_j \geq \frac{\sigma_j}{4\pi} \Delta u_j
\end{equation}
  and, for each $t>0$,
  \begin{equation} \label{e:lc}
    \sigma_j\emph{div}\left(\frac{\nabla u_j}{u_j}\right)(t,\cdot) \geq -\frac{2d\pi}{t}
  \end{equation}
for $j=1,2$ then
  \begin{equation} \label{e:heforu}
    \partial_t u \geq \frac{\sigma}{4\pi} \Delta u
  \end{equation}
  and, for each $t>0$,
  \begin{equation} \label{e:lcforu}
    \sigma\emph{div}\left(\frac{\nabla u}{u}\right)(t,\cdot) \geq -\frac{2d\pi}{t}.
  \end{equation}
\end{theorem}

\begin{corollary} \label{c:Ymon2}
For $j=1,2$ let $u_j$ satisfy the heat equation
\begin{equation*}
    \partial_t u_j = \frac{\sigma_j}{4\pi} \Delta u_j
\end{equation*}
with initial data a compactly supported positive finite Borel
measure. Let $Q : (0,\infty) \rightarrow (0,\infty)$ be given by
\begin{equation*}
  Q(t) := t^{d(\alpha_1+\alpha_2-1-1/p)/2}\left\| u_1(t,\cdot)^{\A_1} * u_2(t,\cdot)^{\A_2}
  \right\|_{L^p(\R^d)}.
\end{equation*}
Then $Q(t)$ is nondecreasing for each $t > 0$.
\end{corollary}

We remark that the idea behind this extension of Theorem
\ref{t:closure}(1) lies in \cite{BCCT} and involves a certain
log-convexity property for solutions to heat equations. In
particular, if there is equality in \eqref{e:he} and the initial
data for $u_j$ is some finite positive Borel measure then
\eqref{e:lc} is automatic by Corollary 8.7 of \cite{BCCT}.

Theorem \ref{t:closure2} self-improves to a result involving
iterated convolutions, as was the case with Theorem \ref{t:closure}.
Again, we leave the details of this to the interested reader.

Finally we remark that all of our results also hold in the setting
of the torus. This will be clear from our proofs.



\section{Proof of Theorem \ref{t:closure}, Corollary \ref{c:Ymon} and Theorem \ref{t:closure2}} \label{section:closure1}
An elementary but crucial component of the proof of Theorem
\ref{t:closure} and Theorem \ref{t:closure2} is the following.

\begin{lemma} \label{l:main}
  Let $0 < \A_1,\A_2,\Lambda_1,\Lambda_2 < \infty$. For $j=1,2$ suppose that $u_j:(0,\infty) \times
\mathbb{R}^d \rightarrow (0,\infty)$ is such that
$u_j(t,\cdot)^{\A_j}$, $\partial_t(u_j(t,\cdot)^{\A_j})$,
$\nabla(u_j(t,\cdot)^{\A_j})$, $u_j(t,\cdot)^{\A_j}|\nabla \log
u_j(t,\cdot)|^2$ and $\Delta(u_j(t,\cdot)^{\A_j})$ are rapidly
decreasing in space locally uniformly in time for $t>0$.
   Then
  \begin{align*}
    & \Lambda_1(u_1^{\A_1} * u_2^{\A_2})(u_1^{\A_1}|\tfrac{\nabla u_1}{u_1}|^2 * u_2^{\A_2}) +
\Lambda_2(u_1^{\A_1} * u_2^{\A_2})(u_1^{\A_1} * u_2^{\A_2}|\tfrac{\nabla u_2}{u_2}|^2) \\
&  + 2\Lambda_1^{1/2}\Lambda_2^{1/2}(u_1^{\A_1} *
u_2^{\A_2})(u_1^{\A_1}\tfrac{\nabla u_1}{u_1} *
  u_2^{\A_2}\tfrac{\nabla u_2}{u_2})- (\tfrac{\Lambda_1}{\A_1^2} + \tfrac{\Lambda_2}{\A_2^2} +
2\tfrac{\Lambda_1^{1/2}\Lambda_2^{1/2}}{\A_1\A_2})|\nabla(u_1^{\A_1}
* u_2^{\A_2})|^2
  \end{align*}
  evaluated at $(t,x) \in (0,\infty) \times \R^d$ coincides
  with
  \begin{align*}
    \tfrac{1}{2} \int_{\R^d} & \int_{\R^d} u_1(t,x-y)^{\A_1}u_2(t,y)^{\A_2}
    u_1(t,x-z)^{\A_1}u_2(t,z)^{\A_2} \times \\
  & |\Lambda_1^{1/2}\tfrac{\nabla u_1}{u_1}(t,x-y) + \Lambda_2^{1/2}\tfrac{\nabla u_2}{u_2}(t,y) -
  \Lambda_1^{1/2}\tfrac{\nabla u_1}{u_1}(t,x-z) - \Lambda_2^{1/2}\tfrac{\nabla u_2}{u_2}(t,z)|^2 \,dydz.
  \end{align*}
\end{lemma}

\begin{proof} We remark that each convolution term is well-defined by the
regularity hypotheses on $u_1$ and $u_2$. Moreover, since
$\nabla(u_j^{\A_j}) = \A_j u_j^{\A_j} \tfrac{\nabla u_j}{u_j}$ is
rapidly decreasing in space for $j=1,2$ it follows that
$\nabla(u_1^{\alpha_1}
* u_2^{\alpha_2})$ coincides with $\alpha_1(u_1^{\alpha_1}
\tfrac{\nabla u_1}{u_1} * u_2^{\alpha_2})$ and
$\alpha_2(u_1^{\alpha_1} * u_2^{\alpha_2}\tfrac{\nabla u_2}{u_2})$,
depending on whether one applies the gradient to the left or right
of the convolution. Thus, Lemma \ref{l:main} follows upon expanding
the square in the integrand and collecting like terms.
\end{proof}

\subsection{Proof of Theorem \ref{t:closure}}
We begin by justifying the closure of the technical regularity
ingredients in Theorem \ref{t:closure}. For $j=1,2$ let $v_j$ be the
time dependent vector field on $\R^d$ given by $v_j := \frac{\nabla
u_j}{u_j}$.

Since the convolution of two rapidly decreasing functions on $\R^d$
is also rapidly decreasing, it is straightforward to check that
$u^{1/p}$ is rapidly decreasing locally uniformly in time. For the
time derivative, we note that
\begin{equation} \label{e:timederiv}
  \partial_t(u^{1/p}) = \partial_t(u_1^{1/p_1}) * u_2^{1/p_2} + u_1^{1/p_1} *
  \partial_t(u_2^{1/p_2}),
\end{equation}
where the interchange of differentiation and integration is
justified since $u_j^{1/p_j}$ and $\partial_t(u_j^{1/p_j})$ are
rapidly decreasing in space locally uniformly in time for $j=1,2$.
Hence $\partial_t(u^{1/p})$ is also rapidly decreasing in space
locally uniformly in time. Similarly, it follows that
$\nabla(u^{1/p})$ is rapidly decreasing in space locally uniformly
in time and, moreover, we may write
\begin{equation} \label{e:gradu1}
  \nabla(u^{1/p}) = \nabla(u_1^{1/p_1}) * u_2^{1/p_2} =
  \tfrac{1}{p_1} u_1^{1/p_1}v_1 * u_2^{1/p_2}
\end{equation}
or, by symmetry, $\nabla(u^{1/p}) = \tfrac{1}{p_2} u_1^{1/p_1}
* u_2^{1/p_2}v_2$. Next, observe that
\begin{equation*}
  u^{1/p}|\nabla \log u|^2 = p^2
  \frac{|\nabla(u^{1/p})|^2}{u^{1/p}} = \tfrac{p^2}{p_1^2}\frac{|u_1^{1/p_1}v_1 *
  u_2^{1/p_2}|^2}{u^{1/p}} \leq \tfrac{p^2}{p_1^2} u_1^{1/p_1}|v_1|^2
  * u_2^{1/p_2}
\end{equation*}
by \eqref{e:gradu1} and the Cauchy-Schwarz inequality. Since
$u_1^{1/p_1}|v_1|^2$ and $u_2^{1/p_2}$ are rapidly decreasing in
space locally uniformly in time by assumption, it follows that
$u^{1/p}|\nabla \log u|^2$ is also rapidly decreasing locally
uniformly in time.

Finally, we note that $\Delta(u^{1/p})$ is rapidly decreasing in
space locally uniformly in time since \eqref{e:gradu1} and our
hypotheses on $u_1$ and $u_2$ imply that
\begin{equation*}
  \Delta(u^{1/p}) = \Delta(u_1^{1/p_1}) * u_2^{1/p_2}. 
\end{equation*}
This concludes our justification of the closure of the regularity
ingredients in Theorem \ref{t:closure}. It is, however, a convenient
opportunity to note that we may also write
\begin{align} \label{e:Deltau1}
    \Delta(u^{1/p}) = \d(\nabla(u_1^{1/p_1}) * u_2^{1/p_2}) = \left\{\begin{array}{llllll}
    \tfrac{1}{p_1^2}u_1^{1/p_1}|v_1|^2 * u_2^{1/p_2} + \tfrac{1}{p_1}u_1^{1/p_1}\d(v_1) *
    u_2^{1/p_2} \\
    \tfrac{1}{p_1p_2} u_1^{1/p_1}v_1 * u_2^{1/p_2}v_2
    \end{array} \right.
  \end{align}
  depending on whether we apply the divergence to the term on the left or right of the convolution.
Since
 \begin{equation*}
   \Delta(u_1^{1/p_1}) = \tfrac{1}{p_1}\d(u_1^{1/p_1}v_1) =
   \tfrac{1}{p_1^2}u_1^{1/p_1}|v_1|^2 + \tfrac{1}{p_1}u_1^{1/p_1}\d(v_1)
 \end{equation*}
 it follows from the regularity of $u_1$ and $u_2$ that each convolution term in \eqref{e:Deltau1} is
well-defined. By symmetry the expression \eqref{e:Deltau1} also holds
with the subscripts 1 and 2 interchanged.

We now turn to proving Theorem \ref{t:closure}(1) where we have $p_j
\geq 1$ and $\partial_t u_j \geq \frac{\sigma_j}{4\pi} \Delta u_j$
for each $j=1,2$. Then,
  \begin{equation*} \label{e:vj}
    \frac{\partial_tu_j}{u_j} \geq \frac{\sigma_j}{4\pi} \frac{\d(\nabla
    u_j)}{u_j} = \frac{\sigma_j}{4\pi} (|v_j|^2 + \d(v_j))
  \end{equation*}
  and therefore, by \eqref{e:timederiv},
  \begin{align*}
    4\pi \frac{\partial_t u}{u^{(p-2)/p}}
   & \geq pu^{1/p}\left(\tfrac{\sigma_1}{p_1}u_1^{1/p_1}|v_1|^2
   * u_2^{1/p_2} + \tfrac{\sigma_2}{p_2}u_1^{1/p_1} * u_2^{1/p_2}|v_2|^2 \right. \\
  & \qquad \qquad \qquad + \left.\tfrac{\sigma_1}{p_1}u_1^{1/p_1}\d(v_1) * u_2^{1/p_2} +
  \tfrac{\sigma_2}{p_2} u_1^{1/p_1} * u_2^{1/p_2}\d(v_2) \right).
 \end{align*}
Furthermore,
\begin{align*}
   \frac{\Delta u}{u^{(p-2)/p}}
   = p(p-1)|\nabla(u^{1/p})|^2 + pu^{1/p}\Delta(u^{1/p})
  \end{align*}
  and therefore,
  \begin{align*}
    -\frac{\sigma_1p_1}{p}\frac{\Delta u}{u^{(p-2)/p}} = -(p-1)\sigma_1p_1|\nabla(u^{1/p})|^2
    + \sigma_1(p-p_1)u^{1/p}\Delta(u^{1/p}) - p \sigma_1u^{1/p}\Delta(u^{1/p}).
  \end{align*}
   Hence, by \eqref{e:Deltau1},
  \begin{align*}
   -\frac{\sigma_1p_1}{p}\frac{\Delta u}{u^{(p-2)/p}} &
   = -(p-1)\sigma_1p_1|\nabla(u^{1/p})|^2 + \tfrac{\sigma_1(p-p_1)}{p_1p_2}u^{1/p} (u_1^{1/p_1}v_1 *
    u_2^{1/p_2}v_2) \\
   & \quad - \tfrac{p\sigma_1}{p_1^2}u^{1/p} (u_1^{1/p_1}|v_1|^2 * u_2^{1/p_2})
   - \tfrac{p\sigma_1}{p_1}u^{1/p} (u_1^{1/p_1}\d(v_1) *
   u_2^{1/p_2}).
  \end{align*}
  By symmetry, it follows that
  \begin{align*}
   -\frac{\sigma_2p_2}{p}\frac{\Delta u}{u^{(p-2)/p}} &
   = -(p-1)\sigma_2p_2|\nabla(u^{1/p})|^2 + \tfrac{\sigma_1(p-p_2)}{p_1p_2}u^{1/p} (u_1^{1/p_1}v_1 *
    u_2^{1/p_2}v_2) \\
   & \quad - \tfrac{p\sigma_2}{p_2^2}u^{1/p} (u_1^{1/p_1} * u_2^{1/p_2}|v_2|^2)
   - \tfrac{p\sigma_2}{p_2}u^{1/p} (u_1^{1/p_1} * u_2^{1/p_2}\d(v_2)).
  \end{align*}
  Thus,
  \begin{align*}
    & \frac{1}{u^{(p-2)/p}}\left[4\pi\partial_t u -
    \tfrac{1}{p}\left(\sigma_1p_1 + \sigma_2p_2\right) \Delta u \right] \\
    & \quad \geq \tfrac{p\sigma_1}{p_1}(1-\tfrac{1}{p_1})u^{1/p}(u_1^{1/p_1}|v_1|^2 *
    u_2^{1/p_2}) + \tfrac{p\sigma_2}{p_2}(1-\tfrac{1}{p_2})u^{1/p}(u_1^{1/p_1} *
    u_2^{1/p_2}|v_2|^2) + \\
    & \quad
    \tfrac{1}{p_1p_2}(\sigma_1(p-p_1) + \sigma_2(p-p_2))
    u^{1/p}(u_1^{1/p_1}v_1 *
    u_2^{1/p_2}v_2) -(p-1)(\sigma_1p_1 + \sigma_2p_2)|\nabla(u^{1/p})|^2.
  \end{align*}
  By Lemma \ref{l:main}, it suffices to verify that
  \begin{equation} \label{e:1}
    p\sigma_1(p_1-1) + p\sigma_2(p_2-1) + 2\left(\tfrac{p\sigma_1}{p_1}(1-\tfrac{1}{p_1})\right)^{1/2}
    \left(\tfrac{p\sigma_2}{p_2}(1-\tfrac{1}{p_2})\right)^{1/2} =
    (p-1)(\sigma_1p_1 + \sigma_2p_2)
  \end{equation}
  and
  \begin{equation} \label{e:2}
   2\left(\tfrac{p\sigma_1}{p_1}(1-\tfrac{1}{p_1})\right)^{1/2}
    \left(\tfrac{p\sigma_2}{p_2}(1-\tfrac{1}{p_2})\right)^{1/2} =
    \frac{1}{p_1p_2}(\sigma_1(p-p_1) + \sigma_2(p-p_2)).
  \end{equation}
  However, \eqref{e:1} and \eqref{e:2} are elementary consequences of the hypotheses
  \eqref{e:scaling} and \eqref{e:BL}. This completes the proof of Theorem
  \ref{t:closure}(1).

  Now suppose that we have $p_j \leq 1$ and $\partial_t u_j \leq \frac{\sigma_j}{4\pi}
\Delta u_j$ for each $j=1,2$. By a very similar argument it follows
that
\begin{align*}
    & \frac{1}{u^{(p-2)/p}}\left[-4\pi\partial_t u + \tfrac{1}{p}\left(\sigma_1p_1 + \sigma_2p_2\right)\Delta u \right] \\
    & \geq \tfrac{p\sigma_1}{p_1}(\tfrac{1}{p_1}-1)u^{1/p}(u_1^{1/p_1}|v_1|^2 *
    u_2^{1/p_2}) + \tfrac{p\sigma_2}{p_2}(\tfrac{1}{p_2}-1)u^{1/p}(u_1^{1/p_1} *
    u_2^{1/p_2}|v_2|^2) \\
    & \quad -
    \tfrac{1}{p_1p_2}(\sigma_1(p-p_1) + \sigma_2(p-p_2))
    u^{1/p}(u_1^{1/p_1}v_1 *
    u_2^{1/p_2}v_2) +(p-1)(\sigma_1p_1 + \sigma_2p_2)|\nabla(u^{1/p})|^2,
\end{align*}
which is nonnegative by \eqref{e:scaling}, \eqref{e:BL} and Lemma
\ref{l:main}. This completes the proof of Theorem
\ref{t:closure}(2).

\subsection{Proof of Corollary \ref{c:Ymon}} Firstly, we verify that
$u_1$ and $u_2$ are sufficiently regular to apply Theorem
\ref{t:closure}. For $j=1,2$ suppose $u_j(0,\cdot) = d\mu_j$ and
that $d\mu_j$ is supported in the euclidean ball of radius $M$
centred at the origin. From the explicit formula for the solution
\begin{equation*}
  u_j(t,x) = \frac{1}{t^{d/2}}\int_{\R^d} e^{-\pi|x-y|^2/t}
  \,d\mu_j(y)
\end{equation*}
it is clear that $u_j(t,x) \leq t^{-d/2}e^{-\pi|x|^2/4t}$ for $|x|
> 2M$ and all $t > 0$. It follows that $u_j^{1/p_j}$ is rapidly
decreasing in space locally uniformly in time. Furthermore, by
interchanging differentiation and integration, it is easy to see
that
\begin{equation} \label{e:vjpoly}
|\nabla u_j(t,x)| \leq 2\pi t^{-d/2}(|x|+M)u_j(t,x).
\end{equation}
Consequently, $\nabla(u_j^{1/p_j})$ and $u_j^{1/p_j}|\nabla \log
u_j|^2$ are rapidly decreasing in space locally uniformly in time.
Similar considerations show that the quantities
$\partial_t(u_j^{1/p_j})$ and $\Delta(u_j^{1/p_j})$ are rapidly
decreasing in space locally uniformly in time.

To complete the proof of Corollary \ref{c:Ymon} it suffices, by the
divergence theorem, to show that for each $t > 0$,
$\int_{R\mathbb{S}^{d-1}} |\nabla u(t,\cdot)|$ tends to zero as $R$
tends to infinity. To see this, note that
$$\nabla u = p\nabla(u^{1/p})u^{1-1/p} =
p(u_1^{1/p_1}v_1 * u_2^{1/p_2})u^{1-1/p}$$ where, as before, $v_1 =
\tfrac{\nabla u_1}{u_1}$. However, if $r$ is chosen such that $1 >
1/r > 1-p$ then, by H\"{o}lder's inequality,
\begin{equation*}
  |\nabla u| \leq p (u_1^{\varepsilon_1
  r'}|v_1|^{r'} * u_2^{\varepsilon_2})^{1/r'} u^{1-1/p + 1/rp}
\end{equation*}
where $\varepsilon_j$ satisfies $(1/p_j - \varepsilon_j)r = 1/p_j$
for $j=1,2$. Any nonnegative power of $u_1$ or $u_2$ is rapidly
decreasing in space and, by \eqref{e:vjpoly}, $|v_1|$ has at most
linear growth in space. Hence, for each $t > 0$, $|\nabla
u(t,\cdot)|$ is rapidly decreasing in space which is clearly
sufficient to see that $\int_{R\mathbb{S}^{d-1}} |\nabla
u(t,\cdot)|$ tends to zero as $R$ tends to infinity.

\subsection{Proof of Theorem \ref{t:closure2}} The verification of the closure of the
regularity properties follows in the same way as in Theorem
\ref{t:closure}. We also remark that, as before, these ingredients
are sufficient to justify all convergence issues related to the
integrals which appear in the proof below.

Let $\beta := d(\A_1+\A_2-1-1/p)/2$ and for $j=1,2$ let $v_j$ denote
the time dependent vector field on $\R^d$ given by $v_j :=
\tfrac{\nabla u_j}{u_j}$. Thus,
  \begin{align*}
    4\pi \frac{\partial_t u}{t^{\beta p}(u_1^{\A_1} * u_2^{\A_2})^{p-2}}
    & \geq p(u_1^{\A_1} * u_2^{\A_2})\left[\A_1\sigma_1u_1^{\A_1}|v_1|^2 * u_2^{\A_2}
  + \A_2\sigma_2u_1^{\A_1} * u_2^{\A_2}|v_2|^2 \right. \\
  & \quad  + \left.\A_1\sigma_1u_1^{\A_1}\d(v_1) * u_2^{\A_2} +
  \A_2\sigma_2 u_1^{\A_1} * u_2^{\A_2}\d(v_2) \right] + \tfrac{4\pi \beta p}{t}(u_1^{\A_1} * u_2^{\A_2})^2.
  \end{align*}
The proof of \eqref{e:heforu} proceeds in a similar way to the proof
of Theorem \ref{t:closure}(1). The difference is that we use some of
the divergence terms on the right hand side of the above inequality
to kill off the factor $\tfrac{4\pi \beta p}{t}(u_1^{\A_1} *
u_2^{\A_2})^2$. In particular, by \eqref{e:lc},
\begin{align*}
  & (1-\rho_1)\A_1\sigma_1(u_1^{\A_1}\d(v_1) * u_2^{\A_2})
  + (1-\rho_2)\A_2\sigma_2(u_1^{\A_1} * u_2^{\A_2}\d(v_2)) + \tfrac{4 \pi \beta}{t}(u_1^{\A_1} *
  u_2^{\A_2})\\
  & \quad \geq \tfrac{2d\pi}{t}(\rho_1 \A_1 + \rho_2 \A_2 - 1 -
  \tfrac{1}{p})(u_1^{\A_1} * u_2^{\A_2})
\end{align*}
and the right hand side of the above inequality vanishes by
\eqref{e:weight}. Now, guided by the proof of Theorem
\ref{t:closure}(1), it follows that
  \begin{align*}
    & \frac{1}{t^{\beta p}(u_1^{\A_1} * u_2^{\A_2})^{p-2}}
    \left[4\pi\partial_t u -
    \tfrac{1}{p}\left(\tfrac{\sigma_1}{\A_1} + \tfrac{\sigma_2}{\A_2}\right)
    \Delta u \right] \\
    & \geq p\A_1(1-\rho_1\A_1)\sigma_1(u_1^{\A_1} * u_2^{\A_2})(u_1^{\A_1}|v_1|^2 *
    u_2^{\A_2}) + p\A_2(1-\rho_2\A_2)\sigma_2(u_1^{\A_1} * u_2^{\A_2})(u_1^{\A_1} *
    u_2^{\A_2}|v_2|^2) \\
    & \quad +
    \A_1\A_2(\sigma_1(p\rho_1-\tfrac{1}{\A_1}) + \sigma_2(p\rho_2 -
    \tfrac{1}{\A_2}))
    (u_1^{\A_1} * u_2^{\A_2})(u_1^{\A_1}v_1 *
    u_2^{\A_2}v_2) \\
    & \quad -(p-1)(\tfrac{\sigma_1}{\A_1} + \tfrac{\sigma_2}{\A_2})|\nabla(u_1^{\A_1} * u_2^{\A_2})|^2.
  \end{align*}
  By \eqref{e:weight}, \eqref{e:BL2} and Lemma \ref{l:main}
  the right hand side of the above inequality is nonnegative. This completes the proof of \eqref{e:heforu}.

  To prove \eqref{e:lcforu} we let $v$ be the time dependent
  vector field on $\R^d$ given by $v := \tfrac{\nabla u}{u}$. Thus,
  \begin{equation*}
    \d(v) = \frac{p}{(u_1^{\A_1} * u_2^{\A_2})^2}\left[\d(\nabla(u_1^{\A_1}
    * u_2^{\A_2}))(u_1^{\A_1} * u_2^{\A_2}) - |\nabla(u_1^{\A_1}
    * u_2^{\A_2})|^2\right].
  \end{equation*}
For $j=1,2$ let
\begin{equation} \label{e:ldefn}
  \lambda_j := \left(\frac{1-\rho_j\alpha_j}{2-\rho_1\A_1 - \rho_2\A_2}\right)^2
\end{equation}
and write
\begin{align*}
  \d(\nabla(u_1^{\A_1} * u_2^{\A_2})) & = \tfrac{\alpha_1}{2}\d(u_1^{\A_1}v_1 *
  u_2^{\A_2}) + \tfrac{\alpha_2}{2}\d(u_1^{\A_1} * u_2^{\A_2}v_2)
  \\
  & = \lambda_1\A_1u_1^{\A_1}\d(v_1) * u_2^{\A_2} +
  \lambda_2\A_2u_1^{\A_1} * u_2^{\A_2}\d(v_2) \\
  & \quad + \lambda_1\A_1^2u_1^{\A_1}|v_1|^2 * u_2^{\A_2} + \lambda_2\A_2^2u_1^{\A_1} *
  u_2^{\A_2}|v_2|^2 \\
  & \quad + \A_1\A_2(1-\lambda_1-\lambda_2)(u_1^{\A_1}v_1 *
  u_2^{\A_2}v_2).
\end{align*}
We shall assume that neither $\sigma_1$ nor $\sigma_2$ is equal to
zero; a simple modification of the argument will handle the
degenerate cases. By \eqref{e:lc},
\begin{align*}
 \lambda_1\A_1u_1^{\A_1}\d(v_1) * u_2^{\A_2} + \lambda_2\A_2u_1^{\A_1} *
 u_2^{\A_2}\d(v_2) & \geq - (\tfrac{\lambda_1 \alpha_1}{\sigma_1} +
 \tfrac{\lambda_1\alpha_1}{\sigma_2})\tfrac{2d\pi}{t}(u_1^{\alpha_1} *
 u_2^{\alpha_2}).
\end{align*}
Moreover, by \eqref{e:BL2} and \eqref{e:ldefn}, it is easy to check
that
\begin{equation*}
  r\sigma(\tfrac{\lambda_1 \alpha_1}{\sigma_1} +
 \tfrac{\lambda_1\alpha_1}{\sigma_2}) = \lambda_1^{1/2} +
 \lambda_2^{1/2} = 1
\end{equation*}
and therefore,
\begin{align*}
\d(v) \geq -\tfrac{2d\pi}{\sigma t} & + \tfrac{r}{(u_1^{\A_1} *
u_2^{\A_2})^2} \left[\lambda_1\A_1^2(u_1^{\A_1} *
u_2^{\A_2})(u_1^{\A_1}|v_1|^2
* u_2^{\A_2}) + \lambda_2\A_2^2(u_1^{\A_1} * u_2^{\A_2})(u_1^{\A_1}
* u_2^{\A_2}|v_2|^2) \right.\\
& + \left. \A_1\A_2(1-\lambda_1-\lambda_2)(u_1^{\A_1} *
u_2^{\A_2})(u_1^{\A_1}v_1 *
  u_2^{\A_2}v_2) - |\nabla(u_1^{\A_1} * u_2^{\A_2})|^2 \right]
\end{align*}
Since $\lambda_1 + \lambda_2 + 2\lambda_1^{1/2}\lambda_2^{1/2} =
(\lambda_1^{1/2} + \lambda_2^{1/2})^2 = 1$ it follows immediately
from Lemma \ref{l:main} that the term in square brackets in the
above inequality is nonnegative. This completes the proof of Theorem
\ref{t:closure2}.

\bibliographystyle{amsalpha}

\end{document}